 \documentclass[12pt,a4wide]{amsart}
\usepackage{amsmath,amsfonts}
\usepackage{amssymb}
\usepackage[all]{xy}
\newtheorem{theorem}{Theorem}[section]
\newtheorem{proposition}[theorem]{Proposition}
\newtheorem{lemma}[theorem]{Lemma}

\newtheorem{remarks}[theorem]{Remarks}
\def\tr{\mathop{\mathbf{tr}}}
\def\det{\mathop{\mathbf{det}}}

\def\beqn{\begin{equation}}
\def\eeqn{\end{equation}}

\def\epf{\qed \enddemo}

\def\fp{\mathfrak p}

\def\a{\alpha}

\def\Claminv2{|C(\Lambda)|^{-2}}
\def\ga{\gamma}

\def\lam{\lambda}

\def\Ome{\Omega}

\def\Aa2D{A^{\a,2}(D)}
\def\bAa2D{\overline{A^{\a,2}(D)}}
\def\Ab2D{A^{\beta,2}(D)}
\def\bAb2D{\overline{A^{\beta,2}(D)}}

\def\Norm#1_#2{\Vert#1\Vert_{#2}}

\def\2pd#1#2{\frac{\partial^2 #1}{\partial #2^2}}
\def\p11d#1#2#3{\frac{\partial^2 #1}{  \partial #2\partial #3  }}

\def\ga{\gamma}
\def\Claminv2{|C(\Lambda)|^{-2}}

\def\lam{\lambda}

\def\ad{\operatorname{ad}}
\def\Ad{\operatorname{Ad}}

\def\exp{\operatorname{exp}}

\def\Aa2D{A^{\a,2}(D)}
\def\bAa2D{\overline{A^{\a,2}(D)}}
\def\Ab2D{A^{\beta,2}(D)}
\def\bAb2D{\overline{A^{\beta,2}(D)}}

\def\ub1#1{\underline{\mathbf 1^{#1}}}

\def\h-g-o-p{hypergeometric orthogonal polynomial }
\def\h-g-o-ps{hypergeometric orthogonal polynomials }

\def\nat0{\mathbb Z_{\ge 0}} % nonnegative  natural numbers

\def\bpf{\begin{proof}}
\def\epf{\end{proof}}
\def\beq{\begin{equation}}
\def\eeq{\end{equation}}

\def\draft{\centerline{(Draft {\the \day}/{\the\month} \the \year.)}}

\def\pia2ta2{\pi({\frac a2})\otimes \overline{\pi(\frac a2)}}

\def\cL#1nu{\mathcal L_{{#1}, \nu}}
\def\cL#1a2{\mathcal L_{{#1}, \frac a2}}
\def\E#1nu{E_{{#1}, \nu}}
\def\E#1a2{E_{{#1}, \frac a2}}

\begin{document}
%\draft
\title[Hua operators and Poisson transform on line bundles]
{Hua operators, Poisson transform and relative discrete series
on line bundle over bounded symmetric domains}
\author{Khalid Koufany}
\address{Khalid Koufany --
Institut {\'E}lie Cartan, 
UMR 7502, Universit{\'e} Henri Poincar{\'e}
    (Nancy 1) B.P. 239, 
F-54506 Vand{\oe}uvre-l{\`e}s-Nancy cedex, 
France
}
\email{khalid.koufany@iecn.u-nancy.fr}
\author{Genkai Zhang}
\address{Genkai Zhang -- Department of Mathematics, Chalmers University of
  Technology and G{\"o}teborg University, S-41296 G{\"o}teborg, Sweden}
\email{genkai@math.chalmers.se}
 
\thanks{\scriptsize \sl 
Research by  G. Zhang 
supported by Swedish Science Council (VR).
}
%\subjclass{????; ????}

\begin{abstract}  
Let $\Omega=G/K$  be a bounded symmetric
domain and $S=K/L$  its Shilov boundary.
We consider the action of $G$ 
on sections
of a homogeneous line bundle over $\Omega$ 
 and the
corresponding eigenspaces of $G$-invariant differential
operators.
The Poisson transform 
maps  hyperfunctions   on  $S$ to the eigenspaces. 
We characterize
the image in terms of twisted Hua operators.
For some special parameters   the Poisson 
transform  is of  Szeg\"o type whose
image is in a relative discrete
series; we compute the corresponding 
elements in the discrete series.
\end{abstract}

\maketitle

\begin{center}
April 2011
\end{center}

\section{Introduction}
Let $X=G/K$ be a Riemannian symmetric space of non compact type. 
It is  known that the Poisson transform maps certain
 parabolically induced representation spaces
into null spaces of some systems of differential equations.
For a minimal parabolic subgroup $P_\text{min}\subset G$,
 Kashiwara {\sl et al.} proved \cite{K et Al.}, that the Poisson transform gives a $G$-isomorphism from the set of hyperfunctions on the maximal boundary $G/P_\text{min}$  onto the joint eigenspace of invariant differential operators on $X$, thus proving the Helgason conjecture \cite{Helgason1}.
We shall be interested
in the case of a Hermitian symmetric space $G/K$
 and the Poisson transform
corresponding a maximal (instead of minimal) parabolic subgroup $P_\text{max}\subset G$
with $G/P$ being the Shilov boundary of $G/K$.
For a certain special parameter
of the induced representation
the image of the transform is a subspace of harmonic functions
on symmetric space $G/K$. 
The precise description
of the image for tube domains
is given in \cite{Johnson, Johnson-Koranyi, Lassalle}
in terms of Hua-harmonic functions 
introduced  earlier by Hua \cite{Hua}.
Its generalization to  non-tube cases
is done by Berline and Vergne \cite{Berline-Vergne}. 
For tube domains with general parameters
 Shimeno  \cite{Shimeno2} 
proved an analogue of Kashiwara {\it et al.} theorem 
for $P_\text{max}\subset G$.
  More precisely, he proved 
that the Poisson transform is a $G$-isomorphism 
from the space of hyperfunctions
 on  the Shilov boundary
 onto the space of eigenfunctions of the Hua operator
 of the second order. \\ The generalization to the non-tube bounded symmetric domains has been given in our earlier paper \cite{Koufany-Zhang}.

A more interesting problem is to consider homogeneous
 line bundles over $\Omega$ with the corresponding weighted action of $G$. 
In this setting Shimeno \cite{Shimeno1} generalized
 the Kashiwara {\it et al.} theorem to homogeneous line bundles over Hermitian symmetric spaces of tube type $G/K$
for minimal parabolic $P_\text{min}\subset G$.
 That is  for a given a line bundle $E_\nu$
(see below)
and for a generic parameter depending on $\nu$
of the induced representation
from a minimal parabolic $P_\text{min}\subset G$.
subgroup,
 the Poisson transform maps
as isomorphisms from hyperfunction-valued 
sections of a line bundle over 
$G/P_\text{min}$
 onto a space of eigenfunctions.
We shall find  characterizations of the Poisson integrals of hyperfunction-valued sections of 
a line bundle over the Shilov boundary 
$G/P_\text{min}$ of a bounded symmetric domain
for generic parameters. 
The Poisson transform becomes more interesting for larger parameters of $\nu$
as there appear relative discrete series, in particular,
the weighted Bergman spaces,
in the Plancherel formula \cite{Shimeno-jfa}, 
the Poisson transform on Shilov boundary for the corresponding parameter
is obviously not injective.  We shall compute
explicitly the image for some of the relative
discrete series. We proceed with some more precise
description of our result.

Let $\Omega=G/K$ be a bounded symmetric domain of tube type
of rank  $r$ and  genus $p$. 
For $\nu\in p\mathbb{Z}$ we consider the
 (unique) character $\tau_\nu$ of $K$ and the corresponding homogeneous line bundle $E_\nu$ over $\Omega$. 
We identify $C^\infty$-sections of $E_\nu$ with the space $C^\infty(G/K,\tau_\nu)$ of $C^\infty$-functions on $G$ such
 that $f(gk)=\tau_\nu(k)^{-1}f(g)$. 
We   consider the generalized Poisson transform 
$(\mathcal{P}_{s,\nu}f)(z)=\int_S P_{s,\nu}(z,u)f(u)du$ where $P_{s,\nu}$
%$$P_{s,\nu}(z,u)=(h(z,z)/|h(z,u)|^2)^{sn/r}h(z,u)^{-\nu}, \; z\in\Omega, u\in S$$ i
is the generalized Shilov kernel and $S=G/P_1$  the Shilov boundary of $\Omega$.  The subgroup $P_1$ is a maximal parabolic subgroup of $G$.
For $s\in\mathbb{C}$,  let $\mathcal{B}(S,s,\nu)$ be the space of hyperfunction-valued sections on $S=G/P_1$ associated with the character of $P_1$ given by 
$man\mapsto e^{(s\rho_0-\rho_1)(\log a)}\tau_\nu(m)$.
 Then for $\lambda_s=\rho+2n(s-1)\xi_e^*-\nu r\xi_e^*$, 
the space $\mathcal{B}(S,s,\nu)$ can be considered as a 
subspace of $\mathcal{B}(G/P;L_{\lambda_s,\nu})$ of hyperfunction-valued sections of a line bundle over  $G/P$. We construct
certain Hua operators on  $G/K$
and  we prove (Theorem \ref{theorem-tube-domains}) that for generic values of $s$ the Poisson transform $\mathcal{P}_{s,\nu}$ is a $G$-isomorphism
between $\mathcal{B}(S,s,\nu)$ and the space of eigenfunctions of the Hua operator of the second order.  

If $\Omega$ is a non-tube type domain it is known
that Hua-type operators of second order will not
be sufficient to characterize the image of Poisson transform
on the Shilov boundary. However for type one non-tube domains
of $r\times (r+b)$-matrices the Lie
algebra  $\mathfrak k_{\mathbb C}$ 
is
a direct sum of $\mathfrak {gl}_{r}$ and
 $\mathfrak {sl}_{r+b}$ and
one can construct \cite{Berline-Vergne}
a second order
Hua operator by taking certain projection
on the summand $\mathfrak {gl}_{r}$.
We prove a corresponding
result for line
bundles in this case; see
 \ref{theorem-type-one-domains}.

For singular value of $s$ we prove (Theorem \ref{Th-rds}) that the Poisson transform is a Szeg\"o type map of principal series representation onto the relative discrete series representation. We compute explicitly the Poisson transform on certain spherical polynomials
on the Shilov boundary.\\

The paper is organized as follows.  In \S2 we  recall very briefly
the Jordan algebraic characterization of bounded symmetric domains. In \S3 we introduce the line bundle over the bounded symmetric domain $\Omega$. The generalized Poisson transform of hyperfunction-valued sections 
on the maximal and the Shilov boundaries are studied in \S4. The characterization of Poisson integrals of 
hyperfunction-valued sections on the Shilov boundary is given in \S5. In this section we also recall our geometric construction of the Hua operator. The necessary condition is proved in \S6. In \S7 we compute the radial part of the Hua system and prove sufficiency condition. Finally in \S8 we show a relationship   between the Poisson transform, Hua operator  and the relative discrete series representation.\\

After a preliminary version of
this paper was finished we were informed by Professor T. Oshima that 
he and N. Shimeno have obtained in \cite{Shimeno-Oshima} 
some similar results about Poisson transforms and Hua operators.
Professor A. Koranyi communicated also his
recent preprint \cite{Koranyi-2011}
to us where he proved the necessity of Theorem 5.2
using different methods. In particular some of the questions
posed in that paper are answered here.

%Let $\Omega$ be a bounded symmetric domain with Shilov boundary $S$.
%The Poisson transform  maps functions on $S$ into eigenspaces of $G$-invariant differential operators on $\Omega$.  The characterization of the Poisson transform is given in the tube case in terms of Hua operator of the second order (Shimeno, Koranyi, Johnson and Koranyi, ..) while in the non tube case this characterization is given by a Hua operator of third order (Beriline and Vergne, Koufany and Zhang \cite{Koufany-Zhang})

%In our earlier paper \cite{Koufany-Zhang} we studied
%the characterization of image  in terms of Hua operators.
%For certain special parameter the image are certain Hua-harmonic
%functions studied earlier by Johnson and Koranyi \cite{Johnson-Koranyi}
%In this paper we extend this studies to the line bundle over $Omega$.

 %A more interesting case  is to consider homogeneous line bundles over $\Omega$ with the corresponding weighted action of $G$;an obvious eigenspace is the Bergman space of holomorphic functions. In the present paper we will the corresponding Hua operators and Poisson transform.

\section{Bounded symmetric domains and Jordan triples}

We begin with a brief review of necessary  facts
on bounded symmetric domains and Jordan triple systems.
Let $V$, $\dim V=n$ be a complex vector space, $\Omega\subset V$ a 
irreducible  bounded symmetric domain. Let $Aut(\Omega)$ be the group of all biholomorphic automorphisms of $\Omega$, let $G$ be the connected component of the identity of $Aut(\Omega)$, and   let   $K$ be the isotropy subgroup of $G$ at the point $0\in\Omega$. As a symmetric space, $\Omega=G/K$. 
 The group $K$ acts as linear
transformations on $V$ and we can thus identify $K$
also as a subgroup of $GL(V)$. 

Let  $\mathfrak{g}$ be  the Lie algebra of $G$ and
$\mathfrak{g}_{\mathbb C}$
its  complexification.
The algebra $\mathfrak{g}
$ has the Cartan decomposition
$\mathfrak{g}=\mathfrak{k}\oplus \mathfrak{p}$
and $\mathfrak{g}_{\mathbb C}
=\mathfrak p^+ +
\mathfrak k_{\mathbb C}
+\mathfrak p^-
$ the Harish-Chandra decomposition.
Denote $Z_0$  the element
 in the center of $\mathfrak k$ which defines the complex structure on
 $\mathfrak{p}^+$,
 i.e., $\ad(Z_0) v= iv$ for $v\in \mathfrak{p}^+$. 
We can thus identify $\mathfrak{p}^+$ with $V$,
$\mathfrak{p}^+=V$.

There exists a quadratic map $Q : V\to End(\bar{V},V)$ (where $\bar{V}$ is the complex conjugate of $V$), such that $\mathfrak{p}=\{\xi_v \; ;\; v\in V\}$
as holomorphic vector fields, where $\xi_v(z):=v-Q(z)\bar{v}$. 
Define $D(z,\bar{v})w:=(Q_{z+w}-Q_z-Q_w)\bar v$. It 
 satisfies  
$$
D(z, \bar{v})w=D(w, \bar{v})z
\; ,\;
 [D(u,\bar{v}),D(z,\bar{w})]=D(\{u\;\bar{v}\;z\},\bar{w})-D(z,\overline{\{w\;\bar{u}\;v\}})
$$
so $V$ is a Jordan triple system. Furthermore we have
$[X,\xi_z]=\xi_{Xz}$ for $X\in\mathfrak{k}$, $z\in V$ and $[\xi_z,\xi_v]=D(z,\bar{v})-D(v,\bar{z})$ for all $z,v\in V$. 
In this realization elements in $\fp^-$ are of the form
$\{-Q(z)\bar v\}$ which we write as $\bar v$. Thus
\begin{equation}
\label{comm}
[v, \bar w]=D(z, \bar w).
\end{equation}
We define
\begin{equation}
\label{scal-norm}
\langle z,w\rangle
 =\frac{1}{p}\tr D(z,\bar{w}),
\end{equation}
where $\tr$ is the trace functional on $End(V)$
and $p$ is the genus defined below. 
It
 is a $K-$invariant Hermitian product on $V$.

The group $K$ acts on $V$ by unitary transformations. The domain $\Omega$ is realized as the open unit ball of $V$ with respect to the spectral norm, 
\begin{equation}
\Omega=\{z\in V\; :\; \|D(z,\bar{z})\|^2<2\},
\end{equation}
where $\|D(z,\bar{z})\|$ is the operator norm of $D(z,\bar{z})$ on the Hilbert space $(V, \langle \cdot, \cdot\rangle)$.

An element $e\in V$ is a tripotent if $\{e\;\bar{e}\; e\}=e$. The subspaces $V_\lambda(e)=\text{ker}(D(e,\bar{e})-\lambda\text{id})$ are called Pierce $\lambda-$spaces. Then we have $V=V_0(e)\oplus V_1(e)\oplus V_2(e)$.
 Two tripotents $e$ and $c$ are orthogonal if $D(e,\bar{c})=0$.  A tripotent $e$ is minimal if it cannot be written as the sum of two non-zero orthogonal tripotents. With the above normalization of inner product
we have $\langle e, e\rangle=1$ for minimal tripotents $e$.
The tripotent $e$ is maximal if $V_0(e)=0$.
 
  A frame is a maximal family of pairwise orthogonal, minimal tripotents. It is known that the group $K$ acts transitively on frames. In particular, the cardinality of all frames is the same, and it is equal to the rank $r$ of $\Omega$. 
  
Let us choose and fix a frame $\{e_j\}_{j=1}^r$ in $V$. Then, by transitivity of $K$ on the frames, each element $z\in V$ admits a polar decomposition $z=k\sum_{j=1}^rs_je_j$, where $k\in K$ and $s_j=s_j(z)$ are the singular numbers of $z$.
Denote  $e$ the maximal tripotent $e=e_1+\ldots+e_r$. The Shilov boundary of $\Omega$ is $S=K/K_1$ where $K_1=\{k\in K\; :\; k \;e=e\}$. It is know that $S$ coincides with the set of maximal tripotents of $V$.
 
 The joint Peirce spaces are
  \begin{equation}
V_{ij}=\{z\in Z\; :\; D(e_k,\bar{e}_k)z=(\delta_{ik}+\delta_{jk})z,\; \forall k\},\; 0\leq i\leq j\leq r
\end{equation}
then $V=\bigoplus_{0\leq i\leq j\leq r}V_{ij}$, $V_{00}=0$, and $V_{ii}=\mathbb{C}e_i$ $(i>0)$.

The  triple of integers $(r,a,b)$ with
 \begin{equation}
a:=\dim V_{jk}\; (1\leq j<k\leq r);\; b:=\dim V_{0j}\; (1\leq j\leq r)
\end{equation}
is independent of the choice of the frame and uniquely determines the Jordan triple. Notice that $b=0$ exactly if $V$ is a Jordan algebra which is equivalent to say that $\Omega$ is of  tube type.

The Peirce decomposition associated with $e$ is then $V=V_2\oplus V_1$ where
\begin{equation}
V_2=\sum_{1\leq j\leq k\leq r} V_{jk}\;\; V_1=\sum_{j=1}^rV_{0j}.
\end{equation}
 
 Let $n_1=\dim V_1$ and $n_2=\dim V_2$, then
 $$n_1=rn,\;\; n_2=r+\frac{r(r-1)}{2}a \;\;\text{and}\;\; n=n_1+n_2.$$
 The genus of $\Omega$ is 
 $$p=p(\Omega)=\frac{1}{r}\tr D(e,\bar{e})=(r-1)a+b+2.$$

%\marginpar{a changer?}

Let $\frak{a}=\mathbb{R}\xi_{e_1}+\ldots +\mathbb{R}\xi_{e_r}$. 
Then   $\mathfrak{a}$ is a maximal abelian subspace of
 $\mathfrak{p}$ with basis  vectors $ \{\xi_{e_1},\ldots ,\xi_{e_r}\}$. 
Its dual basis in $\mathfrak{a}^*$ 
will be denoted by $\{\frac {\beta_j}2\}_{j=1}^r\subset \mathfrak{a}^*$, i.e., 
\begin{equation}
\beta_j(\xi_{e_j})=2\delta_{jk},\;\;1\leq j,k\leq r.
\end{equation}
We define an ordering on $\mathfrak{a}^*$ via
\begin{equation}
\beta_r>\beta_{r-1}>\cdots>\beta_1>0.
\end{equation}
It is known that the restricted roots system $\Sigma(\mathfrak{g},\mathfrak{a})$ of $\mathfrak{g}$ relative to $\mathfrak{a}$ is of type $C_r$ or $BC_r$, it consists of the roots $\pm\beta_j$ $(1\leq j\leq r)$ with multiplicity 1, the roots $\pm\frac{1}{2}\beta_j\pm\frac{1}{2}\beta_k$ $(1\leq j\not=k\leq r)$ with multiplicity $a$, and possibly the roots $\pm\frac{1}{2}\beta_j$ $(1\leq j\leq k)$ with multiplicity $2b$. The set of positive roots 
$\Sigma^+(\mathfrak{g},\mathfrak{a})$ consists of $\frac{1}{2}(\beta_j\pm\beta_k)$ $(1\leq j<k\leq r)$, $\beta_j$ and $\frac{1}{2}\beta_j$ $(1\leq j\leq r)$.

%We will write an element $\underline{\lambda}\in(\mathfrak{a}^*)_\mathbb{C}$ as
%\begin{equation}
%\underline{\lambda}=\sum_{j=1}^r\lambda_j\beta_j,
%\end{equation}
%and identify $\underline{\lambda}$ with $(\lambda_1,\lambda_2,\ldots,\lambda_r)$.
The half sum of positive roots is given by
\begin{equation}
\underline{\rho}=\sum_{j=1}^r\rho_j\beta_j=\sum_{j=1}^r\frac{b+1+a(j-1)}{2}\beta_j.
\end{equation}

Let $A$ be the analytic subgroup of $G$ corresponding to $\mathfrak{a}$. Let $\mathfrak{n}^+=\sum_{\alpha \in \Sigma^+}\mathfrak{g}^\alpha$ and $\mathfrak{n}^-=\theta(\mathfrak{n}^+)$. Let $N^+$ and $N^-$ be the corresponding analytic subgroups of $G$.  Let $M$ be the centralizer of $\mathfrak{a}$ in $K$. The subgroup $P=MAN^+$ is a minimal parabolic subgroup of $G$. 

The set
$\Lambda=\{\alpha_1,\; \ldots,\; \alpha_{r-1},\;
\alpha_r\}$ of
simple roots in $\Sigma^+$ is such that  
\begin{equation*}
\alpha_j=\frac{1}{2}(\beta_{r-j+1}-\beta_{r-j}),\; 1\leq j\leq r-1
\end{equation*}and
\begin{equation*}
\alpha_r=\begin{cases} \beta_1 & \text{for the tube case}\\ \frac{1}{2}\beta_1 & \text{for the non-tube case}.\end{cases}
\end{equation*}
Let $\Lambda_1=\{\alpha_1,\ldots,\alpha_{r-1}\}$ 
%and write
%$\Sigma_1=\Sigma\cap\mathbb{Z} \Lambda_1$. Define
%$$\mathfrak{m}_{1,1}=
%\mathfrak{m} +\mathfrak{a} +\sum_{\gamma\in
 % \Sigma_1}\mathfrak{g}^\gamma,\;\; 
%\mathfrak{n}_1^+=\sum_{\gamma\in
 % \Sigma^+\setminus\Sigma_1}\mathfrak{g}^\gamma.$$
and $P_1$ the corresponding standard parabolic subgroup of $G$ with the Langlands decomposition $P_1=M_1A_1N_1^+$ such that $A_1\subset A$. Then the Lie algebra of $A_1$ is $\mathfrak{a}_1=\mathbb{R}\xi_e$  where $\xi_e=\xi_{e_1}+\ldots+\xi_{e_r}$.\\
The spaces $G/P=K/M$ is the Furstenberg or maximal boundary,
and $G/P_1=K/K_1$, with $K_1=M_1\cap K$,  is the Shilov boundary of $G/K$.

%%%%%%%%%
\section{Line bundle over $\Omega$}
%%%%%%%%%

Denote $Z=\frac{p}{n}Z_0$
where $Z_0$ is the center element
defined in \S2. 
The group $K$ is factorized as $K=\exp(\mathbb{R}Z)K_s$, 
where $K_s$ is the analytic subgroup of $K$ with Lie algebra $\mathfrak{k}_s=[\mathfrak{k},\mathfrak{k}]$.

For a fixed $\nu \in p\mathbb{Z}$  consider the character
 $\tau_\nu$ of $K$ defined by 
$$\tau_\nu(k)=
\begin{cases}
e^{it\nu} & \text{if }\; k=\exp(tZ)
\in\exp(\mathbb{R}Z)\\
1            &\text{if }\; k\in K_s.
\end{cases}$$
In particular we have $J_k(z)^{\frac \nu p} =e^{it\nu}$
for $k=\exp(tZ)$, $z\in \Omega$. Thus
$J_k(z)^{\frac \nu p} =\tau_\nu(k)$,
$k\in K$.
 Here we denote $J_g, g\in G$,  the Jacobian of
the holomorphic mappings $g$ on $\Omega$.
  See \cite{Schlich} and \cite{Dooley-Orsted-Zhang}.\\

%For a real analytic manifold $X$ let $\mathcal{B}(X)$, $\mathcal{A}(X)$ and $C^\infty(X)$ be  the space of all hyperfunctions, analytic functions and smooth functions on $X$ respectively.

%For all $\lambda\in\mathfrak{a}_\mathbb{C}^*$ and $\nu \in p\mathbb{Z}$ and for $\mathcal{F}=\mathcal{B}$, $\mathcal{A}$ or $C^\infty$, we define
%$$\mathcal{F}(G/K, \tau_\nu)=\{F\in \mathcal{F}(G) \; : \; F(gk)=\tau_\nu(k)^{-1} F(g), \;\; g\in G, k\in K\}$$

%%%%%%
%Assume $\nu>p-1$.
 Let $E_\nu$ be the homogeneous line bundle $G\times_{K}
 \mathbb C $ over $G/K=\Omega$, where $K$ acts 
on $\mathbb C$
via the one dimensional representation $\tau_\nu$.
The space $C^\infty(\Omega, E_\nu)$ of smooth sections of $E_\nu$
is by definition
 the space $C^\infty(\Omega, E_\nu)=C^\infty(G/K;\tau_\nu)$
 of $C^\infty$-functions $F$
on $G$ such that
  $$F(gk)=\tau_\nu(k)^{-1}F(g).$$
We will trivialize the bundle via the map $[(g, c)]\in E_\nu
\mapsto (g\cdot 0, J_g(0)^{\frac \nu p}c)\in \Omega \times\mathbb C$
and thus identifies $C^\infty(G/K;\tau_\nu)$ also as
the space of $C^\infty$-functions on $\Omega$ with $G$
acting as
\begin{equation}\label{g-weigt-act}
g\in G:  \quad f(z)\mapsto J_{g^{-1}}(z)^{\frac \nu p} f(g^{-1}z).
\end{equation}

%Let $\mathcal{B}(G/K;\tau_\nu)$ denote the space of hyperfunction valued sections of $E_\nu$ and we identify it with
% $$\mathcal{B}(G/K, \tau_\nu)=\{f\in \mathcal{B}(G) \; ; \; f(gk)=\tau_\nu(k)^{-1} f(g) \; \text{ for } g\in G, k\in K\}.$$
  
 %We will also identify $C^\infty(G/K;\tau_\nu)$ with the space $C^\infty(\Omega;\tau_\nu)$ of 
 
Let $D_\nu(G/K)$ denote the space of $G$-invariant differential operators on $G/K$ acting on  $C^\infty(G/K;\tau_\nu)$. 
We have the Harish-Chandra isomorphism \cite{Shimeno1}
$$\gamma_\nu : D_\nu(G/K)\simeq U(\mathfrak{a})^W$$
where $U(\mathfrak{a})^W$ denote the set of $W$-invariant elements in the enveloping algebra $U(\mathfrak{a})$.

 The characters of $D_\nu(G/K)$ are given by
  %For  $\lambda\in\mathfrak{a}_\mathbb{C}^*$ we define an algebra homomorphism $\chi_{\lambda, \nu}$ of $D_\nu(G/K)$ to $\mathbb{C}$ by
$$\chi_{\lambda, \nu}(D)=\gamma_\nu(D)(\lambda),\;\;\;\;   D\in D_\nu(G/K),\;\; \lambda\in\mathfrak{a}_\mathbb{C}^*.$$

For $\lambda\in\mathfrak{a}_\mathbb{C}^*$ we define  $\mathcal{A}(G/K,\mathcal{M}_{\lambda,\nu})$ to be the space of   functions $\varphi\in C^\infty(G/K;\tau_\nu)$  satisfying the system of differential equations
\begin{equation}\label{efdo}
\mathcal{M}_{\lambda,\nu} \; :\; D \varphi =\chi_{\lambda,\nu} (D) \varphi \;\;\; D\in D_\nu(G/K).
\end{equation}

Finally, let $\mathcal{B}(G/P;L_{\lambda,\nu})$ be the space of hyperfunctions $f$ on $G$ satisfying
$$f(gman)=e^{(\lambda-\rho)(\log a)}\tau_\nu(m)^{-1}f(g)$$
for all $g\in G$, $m\in M$, $a\in A$, $n\in N^+$. The space $\mathcal{B}(G/P;L_{\lambda,\nu})$ is a $G$-submodule of $\mathcal{B}(G)$ and can be identified with the space of hyperfunction valued sections of the line bundle $L_{\lambda,\nu}$ on $G/P$ associated with the character of $P$ given by $man\mapsto e^{(\rho-\lambda)(\log a)}\tau_\nu(m)$, $m\in M$, $a\in A$, $n\in N^+$.

%%%%%%%%%
\section{Poisson transform}
%%%%%%%%%

\subsection{Poisson integrals on  $G/P$ for a minimal $P$} 
Let $\lambda\in\mathfrak{a}^*_\mathbb{C}$ and $\nu\in\mathbb{C}$. We define  the Poisson transform   $\mathcal{P}_{\lambda,\nu}$  by
\begin{equation}
(\mathcal{P}_{\lambda,\nu}f)
(g)=\int_K f(gk)\tau_\nu(k) dk\;\; g\in G
\end{equation}
%for any  $f\in \mathcal{B}(G/P,L_{\lambda,\nu})$, 
for any 
  $f\in \mathcal{B}(G/P,L_{\lambda,\nu})$. % into $\mathcal{B}G/K,\tau_\nu)$.  
  
As elements $f\in \mathcal{B}(G/P,L_{\lambda,\nu})$
are uniquely determined by its restriction to $K$
it is natural to express the integral above as 
on $K$. Indeed, for $g\in G$ denote $\kappa(g)\in K$ and  $H(g)\in\mathfrak{a}$
 to be elements  uniquely determined by   
$$g\in \kappa(g)\exp(H(g))N^+\subset KAN^+=G.$$
 Then
 \begin{equation}
(\mathcal{P}_{\lambda,\nu}f)(g)=\int_K f(k)\tau_\nu(\kappa(g^{-1}k))e^{-(\lambda+\rho)(H(g^{-1}k))}dk
\end{equation}
and maps $\mathcal{B}(G/P,L_{\lambda,\nu})$ into 
$\mathcal{A}(G/K,\mathcal{M}_{\lambda,\nu})$, see \cite[Theorem 5.2]{Shimeno1}\\

The Harish-Chandra $c$ function
$$c_\nu(\lambda)
=\int_{N^-} e^{(\lambda+\rho)(H(\bar{n}))}\tau_\nu(\kappa(\bar{n}))d\bar{n}$$
can be written as
$$
c_\nu(\lambda)=\frac{f_{\nu}(\lambda)}
{e_{\nu}(\lambda)}, 
$$
where  the denominator $e_{\nu}(\lambda)
$ is given, 
in our normalization, by
$$ 
e_\nu(\lambda)=\prod_{j>k} \Gamma(\frac 12 a +\lam_j \pm \lam_k)
\prod_{j} \Gamma(\frac 12(b+1 +2\lam_j +\nu))
\Gamma(\frac 12(b+1 +2\lam_j -\nu));
$$
see \cite{Schlich} and \cite[page 227]{Shimeno1}.

The following theorem characterizes the range of the Poisson transform.
\begin{theorem}[\cite{Shimeno1}, Theorem 8.1]\label{Shimeno-Thm} If $\lambda\in\mathfrak{a}_\mathbb{C}^*$ and $\nu\in\mathbb{C}$ satisfy the conditions
\begin{equation}
-2\frac{\langle \lambda,\alpha\rangle}{\langle \alpha,\alpha\rangle}\notin\{1,2,3,\ldots\} \;\; \text{ for any } \alpha\in\Sigma^+
\end{equation}
\begin{equation}
e_{\nu}(\lambda)\not=0,
\end{equation}
then the Poisson transform $\mathcal{P}_{\lambda,\nu}$ is a $G$-isomorphism of $\mathcal{B}(G/P,L_{\lambda,\nu})$ onto $\mathcal{A}(G/K,\mathcal{M}_{\lambda,\nu})$.
\end{theorem}

\subsection{Poisson integrals on the Shilov boundary} 
 Let $h$ be the unique $K$-invariant polynomial on $\mathfrak{p}^+=V$ whose restriction on $\mathbb{R}e_1\oplus \cdots \oplus \mathbb{R}e_r$ is given by
 $$h(\sum_{j=1}^r t_je_j)=\prod_{j=1}^r(1-t_j^2).$$
 Let $h(z,w)$ be its polarization, i.e.
$h(z,w)$ is holomorphic on $z$ and antiholomorphic in $w$
 such that $h(z,z)=h(z)$.
For any complex number $s$ and for any $\nu$ we define the generalized Poisson kernel
\begin{equation}
P_{s,\nu}(z,u)=\left(\frac{h(z,z)}{|h(z,u)|^2}\right)^{s\frac{n}{r}} h(z,u)^{-\nu},\;\;\; z\in \Omega, u\in S.
\end{equation}

We also define the generalized Poisson transform
\begin{equation}
(\mathcal{P}_{s,\nu}
f
)(z)=\int_S P_{s,\nu}(z,u) f(u)du,\;\; \text{ for } f\in\mathcal{B}(S)
\end{equation}

Let 
$$\mathfrak{a}_1^\top=\sum_{j=1}^r\mathbb{R}(\xi_j-\xi_{j+1})$$
be the orthogonal complement of $\mathfrak{a}_1=\mathbb{R}\xi_e$ in $\mathfrak{a}$ with respect to the Killing form. We denote $\xi_e^*$ the dual vector such that $\xi_e^*(\xi_e)=1$  and we extend $\xi_e^*$ to $\mathfrak{a}$ by the orthogonal projection defined above. Then $\rho_1=\rho_{|_{\mathfrak{a}_1}}$ the restriction of $\rho$ to $\mathfrak{a}_1$ is given by $\rho_1=n\xi_e^*$.
 Define $\rho_0$ to be the linear form on $\mathfrak{a}_1$ such that $\rho_0=r\xi_e^*$.\\

%Let $s\in\mathbb{C}$. 
Consider the following representation of $P_{1}=M_{1}A_{1}N_{1}$ given by  $\sigma_{s,\nu}={\tau_{\nu}}_{|_{M_{1}}}\otimes e^{s\rho_{0}-\rho}\otimes 1$   and let  $\mathcal{B}(G/P_1,s,\nu)$ be the space of hyperfunction valued sections of the line bundle   on $S=G/P_1$ corresponding to $\sigma_{s,\nu}$, i.e., the space of hyperfunctions $f$ on $G$ satisfying
$$f(gman)=e^{(s\rho_0-\rho_1)(\log a)} \tau_{\nu}(m)^{-1} f(g)$$
for all $g\in G$, $m\in M_1$, $a\in A_1$, $n\in N_1^+$.
We may fix $e\in S$ as a base point
and identify also
 $\mathcal{B}(G/P_1,s,\nu)$ with $\mathcal{B}(S)$.

For $s\in\mathbb{C}$ we define  $\lambda_s\in\mathfrak{a}^\ast$ by
\begin{equation}
%\lambda_s=\rho + r(ps+\nu)\xi_e^*.
\lambda_s=\rho+2n(s-1)\xi_e^*-\nu r\xi_e^*
\end{equation}
Under the identification of $C^\infty(G/K;\tau_\nu)$
(and thus its subspace $\mathcal{A}(G/K,\mathcal{M}_{\lambda,\nu})$)
as smooth functions on $\Omega$ we have 
$\mathcal{P}_{\lambda_s,\nu}$ coincides with $\mathcal{P}_{s,\nu}$. 
We omit the routine computations.

Let  $\nu\in p\mathbb{Z}$, then we have
\begin{equation}
\mathcal{B}(G/P_1;s,\nu)\subset \mathcal{B}(G/P;L_{\lambda_s,\nu})
\end{equation}
therefore
\begin{equation}
\mathcal{P}_{s,\nu}(\mathcal{B}(G/P_1;s,\nu))\subset \mathcal{A}(G/K;\mathcal{M}_{\lambda_s,\nu}).
\end{equation}

\section{Hua operators}
The Bergman reproducing kernel of $\Omega$ is $h(z,\bar{z})^{-p}$ 
up to a constant. It is also
 \begin{equation}
h(z,\bar{z})^{-p}=\det B(z,\bar{z})^{-1}
\end{equation}
where
$$B(z,\bar{w})=I-D(z,\bar{w})+Q(z)Q(\bar{w})
$$ 
is the Bergman operator.
Thus $\Omega=G/K$ 
is a K\"ahler manifold with the 
(normalized) Bergman metric
\begin{equation}
  \label{eq:berg-met}
(u,v)_z=\frac 1{p}\partial_u\bar{\partial}_v\log h(z,\bar{z})^{-p}=
\langle B(z,\bar{z})^{-1}u,v\rangle .
\end{equation}
%For any hermitian vector bundle $E$ over $\Omega$, there exists a connection

Let $\tau$ be a finite-dimensional holomorphic representation of $K_\mathbb{C}$. 
Let $E$ be the  Hermitian vector bundle over $\Omega$ associated with $\tau$. 
Then there exists a unique connection operator $\nabla : C^\infty(\Omega,E) \to C^\infty(\Omega, E\otimes T')$ compatible with the Hermitian structure and the anti-holomorphic differentiation, where $T'=T'\Omega$ is the cotangent bundle over $\Omega$. That is, under the splitting in holomorphic and antiholomorphic parts, $T'=T'^{(1,0)}\oplus T'^{(0,1)}$, we have $\nabla=\mathcal{D}+\bar{\partial}$  where $\mathcal{D}$ is the differentiation operator on $E$,
 \begin{equation}
\mathcal{D} : C^\infty(\Omega,E) \to C^\infty(\Omega,E \otimes T'^{(1,0)} ) .
\end{equation}

 Let $\bar{D}$ be the invariant Cauchy-Riemann operator   on $E$ defined in \cite{Englis-Peetre, gz-shimura}, by
 \begin{equation}
\bar{D}f=B(z,\bar{z})\bar{\partial}f.
\end{equation}
Then
 \begin{equation}
\bar{D} : C^\infty(\Omega,E) \to C^\infty(\Omega,E \otimes T^{(1,0)} ) .
\end{equation}
We will use the following identifications, $T'^{(1,0)}_z=\mathfrak{p}^-=V'=\bar{V}$ and $T_z^{(1,0)}=T'^{(0,1)}=\mathfrak{p}^+=V$.\\
%where $T^{(1,0)}$ is the holomorphic tangent bundle and $T^{(0,1)}$ is the antiholomorphic tangent bundle over $\Omega$. We will use the identifications 
 
 We now specialize the above to the line bundle $E_\nu$ associated with the one-dimensional representation $\tau_\nu$. 
The Hua operator $\mathcal H_{\nu}$
is then defined as the resulting operator of the following diagram
$$
  \xymatrix{
    C^\infty(\Omega,E_\nu \otimes\mathfrak{p}^+)
 \ar@{^{}->}[r]^-{\mathcal{D}} \ar@{<-}[d]_{\bar{D}} 
 & C^\infty(\Omega;E_\nu\otimes\mathfrak{p}^+\otimes \mathfrak{p}^-) 
\ar@{->}[d]^{\text{Ad}_{\mathfrak{p}^+\otimes \mathfrak{p}^-}} \\
    C^\infty(\Omega,E_\nu) \ar@{.>}[r]_-{\mathcal{H_\nu}} &
 C^\infty(\Omega,E_\nu\otimes \mathfrak{k}_\mathbb{C})
  }.
  $$
 We may call $\mathcal H=\mathcal H_{\nu}$
the twisted Hua operator to differ it from the  trivial case 
$\nu=0$.

 \begin{remarks}
 \begin{enumerate}
\item We can change the order of $\bar D$
and $\mathcal D$ 
and define another Hua operator,
 $$\mathcal{H'}f= 
\text{Ad}_{\mathfrak{p}^+\otimes \mathfrak{p}^-} \bar{D}\mathcal{D}f,$$
 and one can  prove, by direct computations, the following relation
 \begin{equation}
 (\mathcal{H}-\mathcal{H}')F=-\frac{2n}{r}\nu F Id.
\end{equation}
 \item On other hand, following Johnson-Kor\`anyi \cite{Johnson-Koranyi}, the Hua operator can also be described using the enveloping algebra :  Let $\{E_\alpha\}$  be a basis of $\mathfrak{p}^+$ and $\{E_\alpha^*\}$  be the dual basis of $\mathfrak{p}^+$ with respect to the Killing form. Then the Hua operator is defined as element of $U(\mathfrak{g})_\mathbb{C}\otimes\mathfrak{k}_\mathbb{C}
$ by
 \begin{equation}
\mathcal{H}=\sum_{\alpha,\beta}E_\alpha E_\beta^*\otimes [E_\beta,E_\alpha^*],
\end{equation}
as operator acting from $C^{\infty}(E_{\nu})$
to  $C^{\infty}(E_{\nu}\otimes \mathfrak{k}_\mathbb{C})$.
However the previous definition
using  $\bar D$ and $\mathcal D$
allows direct computation using coordinates on $\Omega$
and hence has some computational (and also conceptual) advantage.
\end{enumerate} 
\end{remarks}

Note  that  $\mathcal{H}$ is by definition $G$-invariant
with respect to the actions of $G$ on two holomorphic bundles, i.e.,
\begin{equation}\label{inv-h}
\mathcal{H} (J_g(z)^{\frac \nu p} 
f(gz))
=J_g(z)^{\frac \nu p} \text{Ad}(dg(z)^{-1})
(\mathcal{H} f)(gz),
\end{equation}
where $\text{Ad}(dg(z)^{-1})
$ stands for the adjoint
action on $\mathfrak k_{\mathbb C}$
of  $dg(z)^{-1}:
 \mathfrak p^+=T_{gz}^{(1, 0)}
\mapsto \mathfrak p^+=T_{z}^{(1, 0)}$
on $\mathfrak p^+
$, via the defining action of $\mathfrak k_{\mathbb C}$
on $\mathfrak p^+$. (Indeed the element
$dg(z)^{-1}$ is in the group $K_{\mathbb C}\subset GL(V)=
GL(\mathfrak p^+)$  via the defining realization
$K\subset GL(V)$, thus $\text{Ad}(dg(z)^{-1})$
acts on $\mathfrak k_{\mathbb C}$; see \cite[Chapt II, Lemma 5.3]{Satake}.)

We state now our main result  for case of
tube domains. The proof is given in the
next two sections.

\begin{theorem}\label{theorem-tube-domains}
Let $\Omega$ be a bounded symmetric domain of tube type. Suppose $s\in \mathbb{C}$ satisfies the following condition
\begin{equation}\label{Shimeno-condition}
\frac{4n(1-s)}{r}\not\in \Lambda_1\cup\Lambda_2
\end{equation}
where $\Lambda_1=\mathbb{Z}_+-2\nu+2$, $\Lambda_2=2\mathbb{Z}_\geq-4\nu+4$. Then the Poisson transform $\mathcal{P}_{\lambda_s,\nu}$ is a $G$-isomorphism of $\mathcal{B}(S;s,\nu)$ onto the space of analytic functions $F$ on $\Omega$ such that
\begin{equation}\label{Hua-system}
\mathcal{H}F=2\frac{n}{r}s(\frac{n}{r}(s-1)+\nu) F  Id .
\end{equation}
\end{theorem}

Here $Id$ stands for the element $-iZ_0$, which
acts on $\fp^+$ as identity.

 \section{The necessary condition of the Hua equations}

The necessity in the above theorem is
a consequence of the following

\begin{theorem}\label{theorem-Hua-Poisson}
For any $u\in S$ the function $z\mapsto P_{s,\nu}(z,\bar{u})$ satisfies the following system of equations
\begin{equation}\label{formula-Hua-Poisson}
\mathcal{H}P_{s,\nu}
= 2\frac{n}{r}s[\frac{n}{r}(s-1)+\nu]P_{s,\nu} Id
\end{equation}
\end{theorem}
  \begin{proof} We prove first the claim for $z=0$, in which case
we have
$$
\mathcal{H} f(0)=\sum_{\alpha, \beta}\partial_{\alpha}\bar \partial_{\beta} f(0) [e_\alpha, \bar e_{\beta}] .
$$
 
Recall the following formulas in \cite[Lemma 5.2]{Koufany-Zhang}: for any 
fixed $\bar{w}\in\bar{V}$ and any complex number $s$,
% the holomorphic and the anti-holorphic differentiation of 
%the function $z\mapsto h(z,\bar{w})^s$ are given by
\begin{equation}\label{deriv-h}
 \bar{\partial} h(w, \bar z)^s=-s  h({w}, \bar z)^s w^{\bar{z}},\;\;\;
\partial h(z,\bar{w})^s=-s  h(z,\bar{w})^s \bar{w}^z
\end{equation}
where
 $$
w^{\bar z}=B(w,\bar{z})^{-1}(w-Q(w)\bar{z}), \;\; 
\bar w^{z}= \overline{w^{\bar z}}
$$
 are called  quasi-inverses 
of $w$ with respect to $\bar{z}$
and $\bar w$ with respect to ${z}$ respectively, viewed
as $(1, 0)$-form and $(0, 1)$-form in terms
of the Hermitian inner product (\ref{scal-norm}).

%Choose an orthonormal basis $\{e_{\alpha}\}$ of $V$ consisting of the fixed frame $\{e_{j}\}_{1\leq j\leq r}$, orthonormal basis of each subspaces $V_{jk}$ and an orthonormal of each subspaces $V_{j0}$. Then
%$$\mathcal{H}P_{s,\nu}(z,u)=\sum_{\alpha,\beta}D(b(z,\bar{z})e_{\alpha},\bar{e}_{\beta})\bar{\partial}_{\alpha}\partial_{\beta} P_{s,\nu}(z,u)$$
Then  we have
$$
  \bar{\partial}P_{s,\nu}(z,\bar{u})= \bar{\partial}
 \left(h(z,u)^{-\nu}\left(\frac{h(z,z)}{|h(z,u)|^2}\right)^{\frac{n}{r}s}
\right)   
   = -({\frac{n}{r}s})P_{s,\nu}(z,\bar{u})[z^{\bar{z}}-u^{\bar{z}}],
 $$
 and
 $$\begin{array}{rl}
 -({\frac{n}{r}s})\partial 
\left(P_{s,\nu}(z,\bar{u})[z^{\bar{z}}-u^{\bar{z}}]\right)= 
&-({\frac{n}{r}s})(\nu+\frac{n}{r}s)P_{s,\nu}(z,\bar{u})\; 
[z^{\bar{z}}-u^{\bar{z}}] \otimes \bar{u}^z
\\
 &+({\frac{n}{r}s})^2P_{s,\nu}(z,\bar{u}) \; 
 [z^{\bar{z}}-u^{\bar{z}}] \otimes\bar{z}^z
\\
 &-({\frac{n}{r}s})P_{s,\nu}(z,\bar{u})\;  \partial [z^{\bar{z}}-u^{\bar{z}}].
 \end{array}
 $$
 Using   (\ref{eq:berg-met})
 the last derivative can be written as
 $$
 \partial [z^{\bar{z}} -u^{\bar{z}}]=\partial(z^{\bar{z}})=\partial\bar{\partial}\log h(z,\bar{z})^{-1}=B(z,\bar{z})^{-1} Id
 $$
 where $Id$ is the identity form in $\mathfrak{p}^+\otimes \mathfrak{p}^-
=\mathfrak{p}^+\otimes (\mathfrak{p}^+)'$.
 Evaluating at $z=0$ we get then, 
by the commutation relation (\ref{comm})
 $$\begin{array}{rl}
 \mathcal{H}P_{s,\nu}(0,\bar{u})=&({\frac{n}{r}s})(\nu+\frac{n}{r}s)D(u,\bar{u})\\
 &+0\\
 &-({\frac{n}{r}s})\sum_{\alpha=1}^n D(e_\alpha,\bar{e}_\alpha).
  \end{array}
 $$
Since $D(u,\bar{u})=2I$ 
and $\sum_{\alpha=1}^n D(e_\alpha,\bar{e}_\alpha)=2\frac{n}{r} I$
by \cite[Lemma 5.1]{Koufany-Zhang},  the claim is proved
for $z=0$. 
Note furthermore that
the Poisson kernel satisfies
$$
P_{s, \nu}(gz, gu)
=J_g(z)^{\frac \nu p}
P_{s, \nu}(gz, gu) \overline{J_g(u)^{\frac \nu p}}.
$$
Thus the claim is true for general $z$ 
by the invariant property 
(\ref{inv-h}) of $\mathcal H$.
\end{proof}
  
\section{The sufficiency condition of the Hua equations} The aim of this section is to prove that, using the radial part of the Hua operator, each solution $F$ of the system (\ref{Hua-system}) satisfies  the system
of equations (\ref{efdo}).
 Then under the condition  
(\ref{Shimeno-condition}) it can  be
proved that  the boundary value of $F$ is contained in  $\mathcal{B}(S;s,\nu)$.\\

%\subsection{The radial part of the Hua opŽrator}
We start to show that eigenfunctions
of the Hua operator (\ref{Hua-system}) are all $\tau_\nu$-spherical
functions. We need
only to prove the claim for $K$-invariant functions $F$ on $\Omega$, i.e.,
$F(kz)=F(z), \, k\in K$.\\

 For that purpose  
we compute the radial part of the Hua operator. 
The functions $F(z)$ will be identified
as permutation invariant and even function $F(t)$
on the diagonal  $z=\sum_{j=1}^r t_j e_j$.
The operator $\mathcal H$ 
has the form
$$
\mathcal H 
F(z)= \sum_{j=1}^r \mathcal H_j F(z)
D(e_j, e_j),
$$
for some operators $\mathcal{H}_j$ in $t=(t_1, \cdots, t_r)$. It's
convenient to find the radial part of $4\mathcal{H}$
(due to the usual convention that $\bar\partial\partial
=\frac 14 (\partial_x^2 +\partial_y^2)=
\frac 14 (\partial_r^2 +\frac 1r \partial_r)$
for radial functions in $z\in \mathbb C$).

 \begin{theorem}\label{radial-part}
Let $\Ome$ be the tube type domain.
 Let $F$ be 
$\mathcal{C}^2$ and  $K-$invariant function, then for $a=\sum_{j=1}^r t_je_j$,
\begin{equation}\label{Hua_radial}
4\mathcal{H}
F(a)=\sum_{j=1}^r\mathcal{H}_jF(t_1,\ldots,t_r)D(e_j,\bar{e}_j),
\end{equation}
where the scalar-valued operators $\mathcal{H}_j$ are given by
$$
\mathcal{H}_j\!\!=\!\!
(1-t_j^2)^2\bigl(\frac{\partial^2}{\partial
  t_j^2}+\frac{1}{t_j}(1+ 2\nu )
\frac{\partial}{\partial t_j}\bigr)+$$
$$+\frac a2 \sum_{k\not=j}(1-t_j^2)(1-t_k^2)\Bigl[\frac{1}{t_j-t_k}\bigl(\frac{\partial
}{\partial t_j}-\frac{\partial }{\partial t_k}\bigr)
+\frac{1}{t_j+t_k}\bigl(\frac{\partial
}{\partial t_j}+\frac{\partial }{\partial t_k}\bigr)\Bigr]+$$
$$
+(-2\nu)(1-t_j^2)\frac{1}{t_j}\frac{\partial
  }{\partial t_j}.$$
\end{theorem}

 \begin{proof}
Recall \cite{gz-shimura} that 
the operator $\mathcal D$ acting on
 $T^{(1, 0)}$-valued
function takes the form
$$
\mathcal{D}F(z)=h^{-\nu}(z, \bar{z}) 
B(z, \bar{z})\sum_k\partial_k (h^{\nu}(z, \bar{z}) B(z, \bar{z})^{-1}F(z))
\otimes  \bar v_k
$$
which is a  $\mathfrak{p}^+\otimes \mathfrak{p}^-$-valued
function.
Thus
\begin{equation*}
\mathcal{H} 
F(z)
=\Ad_{\mathfrak{p}^+\otimes \mathfrak{p}^-} \sum_{k, l}
h^{-\nu}(z, \bar{z}) 
B(z, \bar{z})\sum_k\partial_k
 (h^{\nu}(z, \bar{z}) 
(\bar \partial_{l} F(z))v_l
)
\otimes  \bar v_k,
%\end{split}
\end{equation*}
where we have performed the cancellation of the $B^{-1}$-term in
$\mathcal D$ with $B$ in $\bar D$.
Performing the Leibniz rule
for $\partial_k$ we get the above as sum of two terms, say $I$ and $II$
where
$$
II=\Ad_{\mathfrak{p}^+\otimes \mathfrak{p}^-} \sum_{k, l}
B(z, \bar{z})\sum_k\partial_k
 ( (\bar \partial_{l} F(z)) v_l
)
\otimes  \bar v_k  
$$
 and 
$$
I=\Ad_{\fp\otimes \bar \fp} \sum_{k, l}
h^{-\nu}(z, z) 
B(z, \bar{z})\sum_k (\partial_k h^{\nu}(z, \bar{z}) )
 ( (\bar \partial_{l} F(z))
 v_l
)
\otimes  \bar v_k  
$$
The second order term $II$ is computed as in \cite{Koufany-Zhang} or \cite{Faraut-Koranyi}
and we need only to treat $I$.
Recall that
$$
\partial_{v} h^{\nu}(z, \bar{z})=
\partial_{v} e^{\nu\log h(z, \bar{z})}
=-\nu e^{\nu\log h(z, z)} \partial_{v} (-\log h(z, \bar{z}))
=-\nu
 h^{\nu} (z, \bar{z}) (v, z^{\bar z})
$$
where
$z^{\bar z}=\bar \partial (-\log h(z, \bar{z}))$
is  quasi-inverse of $z$ with respect to $\bar z$;
see \cite{Loos}, \cite{Zhang}. We have then
$$I=-\nu
\sum_{k, l}(v_k, z^{\bar z})
 (\bar \partial_{l} F(z))
[B(z, \bar{z}) v_l,  \bar v_k].
$$
For $z=\sum_{j=1}^r t_j e_j
$ we have
$z^{\bar z}
=\sum_{j=1}^r \frac{t_j}{1-t_j^2} e_j$, and
$$B(z, \bar{z})=\sum_{j=1}^r (1-t_j^2)^2
 D(e_j, \bar{e}_j)$$
which is diagonalized under the Peirce
decomposition of $V$. In particular the above
sum reduces to
$$I=-\nu
\sum_{j=1}^r  \frac{t_j}{1-t_j^2}
 (\bar \partial_{e_j} F(z))
(1-t_j^2)^2 D(e_j, \bar{e}_j).
$$
Being a $K$-invariant function $F(z)$ is
rotation invariant on the plan $\mathbb Ce_j$
and we have then
 $\bar \partial_{e_j} F(z) =\frac 12 \partial_{t_j}F(t)$.
Finally we have
$$I=-\frac{\nu}2
\sum_{j=1}^r  {t_j}(1-t_j^2)
 \partial_{t_j}F(t)
 D(e_j, \bar{e}_j).
$$
To put $I+II$  in a better form we write $t_j(1-t_j^2)=
(1-t_j^2)(\frac{1-t_j^2}{t_j} -\frac{(1-t_j^2)^2}{t_j })$
and we get then the form $\mathcal{H}_j$ as claimed.

\end{proof}

Note that this formula is consistent
with the formula for Laplace-Beltrami operator
on line bundle with parameter $\nu$ where 
the root multiplicity $1$ for the root $\ga_j$
is replaced by $1+2\nu$ and the multiplicity $2b$ 
of  $\frac {\ga_j}2$ by $2b-2\nu$;
see e.g. \cite{Shimeno-jfa}.\\

We prove now  the sufficiency 
of the Hua equation 
(\ref{Hua-system})
in Theorem 5.2.  Let $\varphi_{\lambda_s,\nu}$ the unique elementary spherical function of type $\tau_\nu$ in $E_{\nu}$,
$$\varphi_{\lambda_s,\nu}(g)
=
\int_K e^{(\lambda_s+\rho)H}\tau_\nu(k^{-1}
\kappa(g^{-1}k))dk,\;\; g\in G$$
With some slightly abuse of notation
we denote $\varphi_{\lambda_s,\nu}(z)$ the corresponding
$K$-invariant function on $\Omega$ via our
trivialization.

Suppose $F$ be a $K$-invariant analytic eigenfunction of the Hua equation 
(\ref{Hua-system}). Let $g\in G$. Then the function
$$\Phi_g(z)=\int_K F(gkz)dk
$$
is $K$-invariant solution of the differential system
$$\mathcal{H}_j\Phi=2\frac{n}{r}s(\frac{n}{r}(s-1)+\nu)\Phi,\qquad j=1,\ldots, r.$$
Therefore by a result  of Yan \cite{Yan}, $\Phi_g$ is proportional to $\varphi_{\lambda_s,\nu}$ and thus,
$$\int_K F(gkz)dk=
\varphi_{\lambda_s,\nu}(z) F(g\cdot 0).
$$
By Shimeno  \cite[Theorem 3.2]{Shimeno-jfa} 
we  see that this integral formula characterizes the joint eigenfunctions in $E_\lambda$, that is $F\in \mathcal{A}(G/K,\mathcal{M}_{\lambda_s,\nu})$.
Now applying Theorem \ref{Shimeno-Thm} we can find 
 $f\in \mathcal{B}(G/P,L_{\lambda_s,\nu})$  
such that $F=\mathcal{P}_{\lambda_s,\nu}(f)$.  
The rest
of the proof
is the same as in
 \cite{Berline-Vergne, Koufany-Zhang}, and we 
get,
under the condition (\ref{Shimeno-condition}), that
 $f$ is actually a function on $S$, i.e.,  $f\in\mathcal{B}(S,s,\nu)$.
  We shall not repeat the computations.

 %%%%%%%%%%%%
 
\section{Type one domains}
Let $V=M_{r,r+b}(\mathbb{C})$ be the vector space of complex
$r\times(r+b)-$matrices. $V$ is a Jordan triple system for the following triple product
 $\{x\bar{y}z\}=xy^*z+zy^*x$,
and  the endomorphisms $D(z,\bar{v})$ are given by
\begin{equation*}
D(z,\bar{v})w=\{z\bar{v}w\}=zv^*w+wv^*z.
\end{equation*}
 
Let
\begin{equation*}
\mathbf{I}_{r,r+b}=\{z\in M_{r,r+b}(\mathbb{C}) \; : \; I_r-z^* z\gg 0\}
\end{equation*}
where  $I_r$ denote the unit matrix of rank $r$. Then
$\mathbf{I}_{r,r+b}$ is a bounded symmetric domain of dimension $r(r+b)$, rank $r$ and
genus $2r+b$. The multiplicities are $2b$ and $a=2$ if $2\leq r$, $a=0$
if $r=1$. The domain $\mathbf{I}_{r,r+b}$ is of tube type if and only
if $b=0$. Its Shilov boundary is 
\begin{equation*}
S_{r,r+b}=\{z\in M_{r,r+b}(\mathbb{C}) \; :\; z^* z=I_r\}.
\end{equation*}

As a homogeneous space, the bounded domain $\mathbf{I}_{r,r+b}$ can be identified with $SU(r,r+b)/S(U(r)\times U(r+b))$.  
 
The complex Lie algebra $\mathfrak{k}_\mathbb{C}$ is given by the set of all matrices
$$\left(\begin{matrix} \mathbf{a}& 0\\ 0&\mathbf{d}\end{matrix}\right),\;
\mathbf{a}\in M_{r,r}(\mathbb{C}),\; \mathbf{d}\in M_{r+b,r+b}(\mathbb{C}), \;
\tr(\mathbf{a})+\tr(\mathbf{d})=0.$$
Hence, $\mathfrak{k}_\mathbb{C}$ can be written as the sum 
\begin{equation*} %\label{k-1-2}
\mathfrak{k}_\mathbb{C}=\mathfrak{k}_\mathbb{C}^{(1)}
\oplus\mathfrak{k}_\mathbb{C}^{(2)},
\end{equation*}
where $\mathfrak{k}_\mathbb{C}^{(1)}$ and $\mathfrak{k}_\mathbb{C}^{(2)}$ are the following
ideals %consisting respectively of the matrices 
$$\mathfrak{k}_\mathbb{C}^{(1)}=\{\left(\begin{matrix} \mathbf{a}& 0\\
0&-\frac{\tr(\mathbf{a})}{r+b}I_{r+b}\end{matrix}\right),\;
\mathbf{a}\in M_{r,r}(\mathbb{C})\},$$
%and
$$\mathfrak{k}_\mathbb{C}^{(2)}=\{\left(\begin{matrix} 0& 0\\ 0&\mathbf{d}\end{matrix}\right),\;
\mathbf{d}\in M_{r+b,r+b}(\mathbb{C}),\;\tr(\mathbf{d})=0\}.$$
%Then, identifying
%$\mathfrak{k}_\mathbb{C}$ as linear transformations of $V$,
%we have
%\begin{equation*}
%\mathfrak{k}_\mathbb{C}=\text{span}\{D(u,\bar{v}),\;\; u, v\in V\}, 
%\;\; \mathfrak{k}_\mathbb{C}^{(1)}=\text{span}\{D(u,\bar{v})^{(1)},\;\; u, v\in V\},
%\end{equation*}
%and
%%%\begin{equation*}
%\mathfrak{k}_\mathbb{C}^{(1)}=\text{span}\{D(u,\bar{v})^{(1)},\;\; u, v\in V\},
%\end{equation*}
%where the endomorphism $D(u,\bar{v})^{(1)}$, the projection of  $D(u,\bar{v})^{(1)}$ onto $\mathfrak{k}_\mathbb{C}^{(1)}$, is %given by 
%\begin{equation*}
%D(u,\bar{v})^{(1)}z=uv^*z.
%\end{equation*}
Let $\mathcal{H}^{(1)}$ be the projection of the Hua operator $\mathcal{H}$ on $\mathfrak{k}_\mathbb{C}^{(1)}$. It maps $C^\infty(G/K,\tau_\nu)$ into $C^\infty(G/K,\tau_\nu\otimes \mathfrak{k}_\mathbb{C}^{(1)})$.\\

The main result of this section is the following   theorem which is the analogue of Theorem \ref{theorem-tube-domains}  for non tube domains of type one.
\begin{theorem}\label{theorem-type-one-domains}
Let $\Omega=\mathbf{I}_{r,r+b}$ be a bounded symmetric domain of type one. Suppose $s\in \mathbb{C}$ satisfies the condition (\ref{Shimeno-condition}). Then the Poisson transform $\mathcal{P}_{\lambda_s,\nu}$ is a $SU(r,r+b)$-isomorphism of $\mathcal{B}(S_{r,r+b};s,\nu)$ onto the space of analytic functions $F$ on $\mathbf{I}_{r,r+b}$ such that
\begin{equation}\label{Hua-condition-type-one}
\mathcal{H}^{(1)}F=(r+b)s((r+b)(s-1)+\nu) F \otimes I_r .
\end{equation}
\end{theorem}

The bounded domain $\mathbf{I}_{r,r+b}$ is of non tube type, however the characterization of Poisson integrals involves a Hua operator  of the second order. As a consequence, the proof of Theorem \ref{theorem-type-one-domains} is similar to Theorem \ref{theorem-tube-domains} and \cite[Theorem 6.1]{Koufany-Zhang}. \\
The necessarily condition is guaranteed by the following 

\begin{proposition} For any $u\in S_{r,r+b}$ the function $z\mapsto P_{s,\nu}(z,\bar{u})$ satisfies
$$\mathcal{H}P_{s,\nu}(z,\bar{u})=(r+b)s[(r+b)(s-1)+\nu]P_{s,\nu}(z,\bar{u}) I_r.$$
\end{proposition}
Indeed, using the computations of the proof of Theorem \ref{theorem-Hua-Poisson} and some arguments in the proof of  \cite[Theorem 5.3]{Koufany-Zhang} we can extend the formula (\ref{formula-Hua-Poisson}) to any (not necessarily tube) bounded symmetric domain:
$$\begin{array}{rl}
\mathcal{H}^{1}P_{s,\nu}(z,\bar{u})=&\bigl[ (\frac{n}{r}s)^2 D(B(z,\bar{z})(\bar{z}^z-\bar{u}^z),z^{\bar{z}}-u^{\bar{z}}) -(\frac{n}{r}s)pZ_0\\
                                                  -&(\frac{n}{r}s)\nu D(B(z,\bar{z})(\bar{u}^z-\bar{u}^z),z^{\bar{z}}-u^{\bar{z}})   \bigr] P_{s,\nu}(z,\bar{u}) I_r .
\end{array}$$

Specifying this formula to the domain $\mathbf{I}_{r,r+b}$ we get proposition.\\

 On the other hand, to prove the   sufficiency condition, the key point is to compute the radial part of the Hua operator $\mathcal{H}^{(1)}$. This follows immediately from the proof of Theorem \ref{radial-part}.
 \begin{theorem} Let $f$ be 
$\mathcal{C}^2$ and  $K-$invariant function, then for $a=\sum_{j=1}^r t_je_j$,
\begin{equation}\label{Hua_radial}
4\mathcal{H}^{(1)}f(a)=\sum_{j=1}^r\mathcal{H}_jF(t_1,\ldots,t_r)
D(e_j,\bar{e}_j)^{(1)},
\end{equation}
where the scalar-valued operators $\mathcal{H}_j$ are given by
$$
\mathcal{H}_j\!\!=\!\!
(1-t_j^2)^2\bigl(\frac{\partial^2}{\partial
  t_j^2}+\frac{1}{t_j}(1+ 2\nu )
\frac{\partial}{\partial t_j}\bigr)+$$
$$+ \sum_{k\not=j}(1-t_j^2)(1-t_k^2)\Bigl[\frac{1}{t_j-t_k}\bigl(\frac{\partial
}{\partial t_j}-\frac{\partial }{\partial t_k}\bigr)
+\frac{1}{t_j+t_k}\bigl(\frac{\partial
}{\partial t_j}+\frac{\partial }{\partial t_k}\bigr)\Bigr]+$$
$$
+(2b-2\nu)(1-t_j^2)\frac{1}{t_j}\frac{\partial
  }{\partial t_j}$$
and $D(e_j,\bar{e}_j)^{(1)}$
is the $\mathfrak{k}_\mathbb{C}^{(1)}$ component
of $D(e_j,\bar{e}_j) \in
\mathfrak{k}_\mathbb{C}$.
\end{theorem}

\section{Relative discrete series}

The Poisson transform is not injective 
for singular $s$ being 
 in the set $
 \Lambda_1\cup\Lambda_2$ in (\ref{Shimeno-condition}).
It arises thus
a question of understanding the image. A finite
subset of  such $s$ corresponds to the
 relative discrete series, i.e. the images
constitute  discrete components
in the decomposition of the $L^2$ space
for the bundle. In  this final section we find
the precise parameters and compute
explicitly the corresponding Poisson transforms
on some distinguished functions, thus
producing elements in the relative
discrete series.

Fix $\nu>p-1$ and $\nu \in p\mathbb Z$. Let $\alpha =\nu-p> -1$
and consider the weighted probability measure on $\Omega$
\begin{equation}
d\mu_\alpha (z) = c_\alpha h(z,\bar{z})^{\alpha}  dm(z)
\end{equation}

The group $G$ acts unitarily on  $L^2(\Omega,\mu_\alpha)$ via
(\ref{g-weigt-act}).

The irreducible decomposition of $L^2(\Omega,\mu_\alpha)$  under the $G$-action has been given by Shimeno in \cite[Theorem 5.10]{Shimeno-jfa} where he proved  abstractly that all discrete parts called relative discrete series appearing in the decomposition are holomorphic discrete series.
 In this section we need their explicit realization given by the second author in \cite{Zhang}.\\
 
 Let us introduce the conical functions, see \cite{Faraut-Koranyi}. Let $c_1,\ldots,c_r$ be the fixed Jordan frame. Put $e_j=c_1+\ldots+c_j$, for $j=1,\ldots, r$. Let $U_j=\{z\in V\; :\; D(c_j,c_j)z=z\}$. Then $U_j$ is a Jordan subalgebra of $V_1=V_1(e)$ with a determinant polynomial $\Delta_j$. We extend the principal minors $\Delta_j$ to all $V$ via $\Delta_j(z):=\Delta_j(P_{U_j}(z))$, where $P_{U_j}$ is the orthogonal projection onto $U_j$. Notice that $\Delta_r(z)=\Delta(z)=\det(z)$.\\
 
 For any $\underline{m}=(m_1,\ldots,m_r)\in\mathbb{C}^r$, consider the associated conical function
 $$\Delta_{\underline{m}}(z):=\Delta_1^{m_1-m_2}(z)\Delta_2^{m_2-m_3}(z)\cdots \Delta_r^{m_r}(z).$$
 If $z=\sum_{j=1}^rz_j c_j$ then $\Delta_{\underline{m}}(z)=\prod_{j=1}^rz_j^{m_j}$.\\

Denote
$$\ell=\begin{cases} 
  \frac{\alpha+1}{2}-1=\frac{\nu-p-1}{2}  & \mbox{if $\alpha$ is an odd integer} \\
  [\frac{\alpha+1}{2}]=[\frac{\nu-p+1}{2}]& \mbox{otherwise }
\end{cases}$$
here $[t]$ stands for the integer part of $t\in\mathbb{R}$. Define
\begin{equation}
D_\nu=\{\underline{m}=\sum_{j=1}^r m_j\gamma_j,\;\; 0\leq m_1\leq \cdots \leq m_r\leq \ell,\; m_j\in\mathbb{Z} \}.
\end{equation}
We let $A^{2,\alpha}_{\underline{m}}$ to be the subspace of $L^2(\Omega,\mu_\alpha)$ generated by the function $\bar{\Delta}_{\underline{m}}(q(z))$, for $\underline{m}\in D_\nu$, where
$$
q(z)=\bar{z}^z
$$ is
quasi-inverse of $\bar z$ with respect to $z$.

We reformulate \cite[Theorem 5.10]{Shimeno-jfa} and \cite[Theorem 4.7, remark 4.8]{Zhang} in the following.
\begin{theorem}
The relative discrete series representations
appearing in $L^2(\Omega,\mu_\alpha)$ are all holomorphic discrete series of the form $A^{2,\alpha}_{\underline{m}}$ with $\underline{m}\in D_\nu$.  The highest weight vector of $A^{2,\alpha}_{\underline{m}}$ is given by $\bar{\Delta}_{\underline{m}}(q(z))$.\\
\end{theorem}

We can now state the main theorem of this section.
 
\begin{theorem}\label{Th-rds} Let $\underline{\delta}=(\delta,\delta,\ldots,\delta)$ such that $s\frac{n}{r}=\frac{n}{r}+\delta-\nu$ and $0\leq \delta<[\frac{\nu-p}{2}]$. Then the Poisson transform $\mathcal{P}_{\lambda_s,\nu}$ is a $G$-equivariant Szge\"o type map  from  the space $\mathcal{B}(S,s,\nu)_K$
of $K$-finite elements
 onto the $K$-finite elements
in the relative discrete series  $A^{2,\alpha}_{\underline{\delta}}$.
\end{theorem}

This theorem is consequence of the following proposition.
  
\begin{proposition}\label{Th-rds}
Under the same conditions as in Theorem \ref{Th-rds} 
%$\underline{\delta}=(\delta,\delta,\ldots,\delta)$ such that $s\frac{n}{r}=\frac{n}{r}+\delta-\nu$ and $0\leq \delta<[\frac{\nu-p}{2}]$. Then 
we have
\begin{equation}
(\mathcal{P}_{s,\nu}\bar{\Delta}_{\underline{\delta}})(z)=
\int_{S} P_{s,\nu}(z,u) \bar{\Delta}_{\underline{\delta}}(u) du=\frac{(s\frac{n}{r})_\delta}{(\frac{n}{r})_\delta} \bar{\Delta}_{\underline{\delta}}(q(z)).\\
\end{equation}
\end{proposition}
 
 Let us first prove the formula below
 
 \begin{lemma}\label{Formula-Delta-K} For any highest weight $\underline{m}$ the following formula holds
 \begin{equation}
\Delta(z)^\delta\,\bar{\Delta}(w)^\delta\,K_{\underline{\mathbf{m}}}(z,w)=c(\underline{\mathbf{m}},\delta)\,K_{\underline{\mathbf{m}}+\delta}(z,w)
\end{equation}
where $c(\underline{\mathbf{m}},\delta)=(\frac{n}{r}+\underline{\mathbf{m}})_
{\underline \delta}:=\prod_{j=1}^r(\frac{a}{2}(r-j)+1+m_j)_\delta$
\end{lemma}

\begin{proof}
Since  $f(z)\mapsto \Delta(z)^\delta f(z)$ is a K-intertwining map from 
$\mathcal{P}_{\underline{\mathbf{m}}}$ onto $\mathcal{P}_{\underline{\mathbf{m}}+\delta}$ we have that $\Delta(z)^\delta\,\overline{\Delta(w)^\delta}\,K_{\underline{\mathbf{m}}}(z,w)$ is equal to $K_{\underline{\mathbf{m}}+\delta}(z,w)$ up to a constant $c(\underline{\mathbf{m}},\delta)$. Now taking $z=w=e$ and using \cite[Lemma 3.1, Theorem 3.4]{Faraut-Koranyi} we find that the constant is 
$$
c(\underline{\mathbf{m}},\delta) =\frac{K_{\underline{\mathbf{m}}}(e,e)}{K_{\underline{\mathbf{m}}+\delta}(e,e)}      
      =\frac{d_{\underline{\mathbf{m}}}}{({n}/{r})_{\underline{\mathbf{m}}}} \frac{({n}/{r})_{\underline{\mathbf{m}}+\delta}} {d_{\underline{\mathbf{m}}+\delta}}  
    =  (\frac{n}{r}+\underline{\mathbf{m}})_\delta\, .
$$
%since $d_{\underline{\mathbf{m}}}=d_{\underline{\mathbf{m}}+\delta}$.
%
% \marginpar{why $d_{\underline{m}}=d_{\underline{m}+\delta}$}
 %
\end{proof}
 \begin{proof}[Proof of Proposition \ref{Th-rds}] Put $\sigma=s\frac{n}{r}$. We compute the image $\mathcal{P}_{s,\nu}\bar{\Delta}_{\underline{\delta}}$,  
\begin{equation*}
\begin{split}
  (\mathcal{P}_{s,\nu}\bar{\Delta}_{\underline{\delta}})(z)=&\int_S h(z,u)^{-\nu}h(z,z)^{\sigma}h(z,u)^{-\sigma}h(u,z)^{-\sigma}   \bar{\Delta}_{\underline{\delta}}(u) du \\
  =&h(z,z)^\sigma \int_S h(z,u)^{-\nu-\sigma} h(u,z)^{-\sigma}  \bar{\Delta}_{\underline{\delta}}(u) du
\end{split}
\end{equation*}
 We use the Faraut-Koranyi expansion \cite[Theorem 3.8]{Faraut-Koranyi-90} of the reproducing kernels $h(z,u)^{-\nu-\sigma}$ and  $h(u,z)^{-\sigma}$ so that,
\begin{equation*}
\begin{split}
 (\mathcal{P}_{s,\nu}\bar{\Delta}_{\underline{\delta}})(z)=
 & h(z,z)^\sigma
  \int_S 
  \bigl[ 
  \sum_{\underline{\mathbf{m}}\geq 0}(\nu+\sigma)_{\underline{\mathbf{m}}} K_{\underline{\mathbf{m}}}(z,u) \times \\
   & \hfill  \times 
   \sum_{\underline{\mathbf{m}}'\geq 0}(\sigma)_{\underline{\mathbf{m}}'} K_{\underline{\mathbf{m}}'}(u,z) 
   \bigr ]
   \bar{\Delta}_{\underline{\delta}}(u) du
\end{split}
\end{equation*}
 which by  Lemma \ref{Formula-Delta-K}  is
\begin{equation*}
\begin{split}
  (\mathcal{P}_{s,\nu}\bar{\Delta}_{\underline{\delta}})(z)
   =& h(z,z)^\sigma
  \int 
  \bigl[ 
  \sum_{\underline{\mathbf{m}}\geq 0}(\nu+\sigma)_{\underline{\mathbf{m}}} K_{\underline{\mathbf{m}}+\delta}(z,u) \Delta_{\underline{\delta}}(z)^{-1} c(\underline{\mathbf{m}},\delta) \times \\
   & \hfill  \times 
   \sum_{\underline{\mathbf{m}}'\geq 0}(\sigma)_{\underline{\mathbf{m}}'} K_{\underline{\mathbf{m}}'}(u,z) 
   \bigr ]
     du\\
   =& h(z,z)^\sigma 
   \sum_{\underline{\mathbf{m}}\geq 0} 
   \bigl[ 
   (\nu+\sigma)_{\underline{\mathbf{m}}} (\sigma)_{\underline{\mathbf{m}}+\delta}  \, c(\underline{\mathbf{m}},\delta) \Delta_{\underline{\delta}}(z)^{-1} \times \\
   &\hfill   \times \int K_{\underline{\mathbf{m}}+\delta}(z,u) K_{\underline{\mathbf{m}}+\delta}(u,z)   du
     \bigr ]
\end{split}
\end{equation*}
The last equality follows from the Schur orthogonality relation. Furthermore, since the ratio  of the Fischer inner product and the standard $K$-invariant inner product of $L^2(S)$ is constant, see \cite[Corollary 3.5]{Faraut-Koranyi} or \cite{Upmeier-Toeplitz}, we obtain
\begin{equation*}
\begin{split}
  (\mathcal{P}_{s,\nu}\bar{\Delta}_{\underline{\delta}})(z) =& h(z,z)^\sigma  
  \sum_{\underline{\mathbf{m}}\geq 0} 
  \bigl[ 
   (\nu+\sigma)_{\underline{\mathbf{m}}} (\sigma)_{\underline{\mathbf{m}}+\delta}  \, c(\underline{\mathbf{m}},\delta) \Delta_{\underline{\delta}}(z)^{-1} \times \\
    & \hspace{3cm} \times  (\frac{n}{r})^{-1}_{\underline{\mathbf{m}}+\delta} K_{\underline{\mathbf{m} }+\delta}(z,z) 
    \bigr ]\\
\end{split}
\end{equation*}
    which again by Lemma \ref{Formula-Delta-K} is
\begin{equation*}
\begin{split}
  (\mathcal{P}_{s,\nu}\bar{\Delta}_{\underline{\delta}})(z) 
  =& h(z,z)^\sigma \bar{\Delta}_{\underline{\delta}}(z)
   \sum_{\underline{\mathbf{m}}\geq 0}   
    (\nu+\sigma)_{\underline{\mathbf{m}}} (\sigma)_{\underline{\mathbf{m}}+\delta}   (\frac{n}{r})^{-1}_{\underline{\mathbf{m}}+\delta} \times \\
  & \hspace{3cm} \times  \bar{\Delta}_{\underline{\mathbf{\delta}}}(z) K_{\underline{\mathbf{m}}}(z,z)\\
   =& h(z,z)^\sigma  
   \frac{(\sigma)_{\delta}}{(\frac{n}{r})_{\delta } } \bar{\Delta}_{\underline{\delta}}(z)
    \sum_{\underline{\mathbf{m}}\geq 0}  
    (\sigma+\delta)_{\underline{\mathbf{m}}}  K_{\underline{\mathbf{m}}}(z,z)
     \\
      =& h(z,z)^\sigma  
   \frac{(\sigma)_{\delta}}{(\frac{n}{r})_{\delta } } \bar{\Delta}_{\underline{\delta}}(z) h(z,z)^{-(\sigma+\delta)}\\
   =&  \frac{(\sigma)_{\delta}}{(\frac{n}{r})_{\delta } } \frac{\bar{\Delta}_{\underline{\delta}}(z)}{h(z,z)^\delta}\\
   %=& \frac{(\sigma)_{\delta}}{(\frac{n}{r})_{\delta } }  \bar{\Delta}_{\underline{\delta}}(q(z)).
\end{split}
\end{equation*}
Finally since $\frac{\bar{\Delta}(z)}{h(z,z)}=\bar{\Delta}(q(z))$, see \cite[Corollary 4.4]{Zhang}, the proof is completed.

\end{proof}

 %%%%%%%%%%%% %%%%%%%%%%%%
 %%%%%%%%%%%% %%%%%%%%%%%%
 %%%%%%%%%%%% %%%%%%%%%%%%
 %%%%%%%%%%%% %%%%%%%%%%%%
\bibliographystyle{amsplain}

\end{document}